\documentclass[12pt]{amsart}

\usepackage[utf8]{inputenc}
\usepackage{amsmath} 
\usepackage{amsfonts}
\usepackage{amssymb}
\usepackage{mathtools}
\usepackage{tikz}
\usepackage{float}
\usepackage{hyperref}
\usepackage{algorithm}
\usepackage{algpseudocode}

\theoremstyle{definition}
\newtheorem{theorem}{Theorem}[section]
\newtheorem*{theorem*}{Theorem}
\newtheorem{proposition}[theorem]{Proposition}

\newtheorem{lemma}[theorem]{Lemma}
\newtheorem{alg}[theorem]{Algorithm}
\newtheorem{corollary}[theorem]{Corollary}
\newtheorem{conjecture}[theorem]{Conjecture}
\newtheorem{definition}[theorem]{Definition}
\newtheorem{example}[theorem]{Example}

\theoremstyle{remark}
\newtheorem{remark}[theorem]{Remark}

\newcommand{\rgf}{\text{rgf}}
\newcommand{\inc}{\text{inc}}
\newcommand{\pl}{\text{plane}}
\newcommand{\west}{\omega}
\newcommand{\east}{\varepsilon}
\newcommand{\sgn}{\text{sgn}}
\newcommand{\inv}{\text{Inv}}
\newcommand{\rk}{\text{rk}}
\makeatletter
\@namedef{subjclassname@2020}{\textup{2020} Mathematics Subject Classification}
\makeatother

\begin{document}

\title[On an alternating sum of factorials and Stirling numbers of the first kind]{On an alternating sum of factorials and Stirling numbers of the first kind: trees, lattices, and games}

\author{Victor Wang}
\address{
 Department of Mathematics,
 Harvard University,
 Cambridge MA 02138, USA}
\email{vwang@math.harvard.edu}


\begin{abstract} 
We study an alternating sum involving factorials and Stirling numbers of the first kind. We give an exponential generating function for these numbers and show they are nonnegative and enumerate the number of increasing trees on $n$ vertices that are won by the second of two players when interpreted as a game tree. We also give a simple description of the quotient from the weak order to the Tamari lattice in terms of plane trees, and give bijections between plane trees, $213$-avoiding permutations, and $312$-avoiding permutations. Finally, for a rooted tree, we give equivalent characterizations of when it describes a game won by the first or second player in terms of the rank-generating function of the lattice of prunings and the Euler characteristic of an associated real variety.
\end{abstract}

\maketitle
\tableofcontents

\section{Introduction}\label{sec:intro}  

Our original motivation for studying the mathematics encountered in this paper stems from a problem involving symmetric functions in noncommuting variables. In 2006, Bergeron, Hohlweg, Rosas, and Zabrocki \cite{BHRZ} defined the $\mathbf x$-basis for symmetric functions in noncommuting variables, motivated by the representation theory of partition lattice algebras. In a talk from the same year, Zabrocki \cite{Zabrocki} observed that the $\mathbf x$-basis appeared to expand with either all nonnegative or all nonpositive coefficients in the $\mathbf e$-basis for symmetric functions in noncommuting variables.

The expansion in the $\mathbf e$-basis can be computed by combining \cite[Equation 4.2]{BHRZ} and \cite[Theorem 3.4]{RosasSagan}:
$$(-1)^{m-1}\mathbf x_{[m]}=\sum_{\tau\ge\sigma}(-1)^{\ell(\sigma)-\ell(\tau)} \frac{(\ell(\tau)-1)!(\lambda(\sigma,\tau)-1)!}{(\lambda(\tau)-1)!}\mathbf e_\sigma.$$
Here, some notation needs to be explained. A \textit{set partition} $\sigma$ of $[m]=\{1,\dots,m\}$ is a family of disjoint sets whose union is $[m]$, and we write $\sigma\vdash [m]$. The sum is over all pairs of set partitions $\tau,\sigma$ of $[m]$ for which every block of $\sigma$ is contained in some block of $\tau$. For a set partition $\sigma$, $\ell(\sigma)$ counts the number of parts of $\sigma$, and $(\lambda(\sigma)-1)!$ is the product of the factorials of one less than the number of elements in each block of $\sigma$. Given $\tau\ge \sigma$, $(\lambda(\sigma,\tau)-1)!$ is the product of the factorials of one less than the number of blocks of $\sigma$ contained in each block of $\tau$.

Zabrocki's observation can be reformulated as the conjecture stated next.
\begin{conjecture}\cite{Zabrocki} \label{conj:pos}
For set partitions $\sigma$, the numbers
$$\sum_{\tau\ge\sigma}(-1)^{\ell(\sigma)-\ell(\tau)} \frac{(\ell(\tau)-1)!(\lambda(\sigma,\tau)-1)!}{(\lambda(\tau)-1)!}$$
are nonnegative.
\end{conjecture}

In the foreword to later editions of P\'olya's \textit{How to Solve It} \cite{Polya} Conway attributes the following advice to P\'olya: ``If you can’t solve a problem, then there is an easier problem you can’t solve: find it."

For the combinatorially-minded, to show the nonnegativity of a number often involves showing that the number enumerates a class of combinatorial objects. Part of the difficulty in proving Conjecture~\ref{conj:pos}, then, is the appearance of denominators in the alternating sum. We therefore will instead study the signs of the numbers
$$\sum_{\tau\ge\sigma}(-1)^{\ell(\sigma)-\ell(\tau)} (\ell(\tau)-1)!(\lambda(\sigma,\tau)-1)!$$
for fixed $\sigma$, where we have removed the denominators, and hope to find a combinatorial interpretation. Our first observation is that this sum depends only on the number of parts $n=\ell(\sigma)$ of $\sigma$, and may be rewritten as
$$\sum_{\tau\vdash[n]}(-1)^{n-\ell(\tau)}(\ell(\tau)-1)!(\lambda(\tau)-1)!.$$

Although the problem in this paper stemmed from a problem in symmetric functions in noncommuting variables, we will see that the mathematics involved quickly diverges from the mathematics of our original motivation. We will show in Theorem~\ref{thm:pos} that the numbers 
$$a_n=\sum_{\tau\vdash[n]}(-1)^{\ell(\tau)-1}(\ell(\tau)-1)!(\lambda(\tau)-1)!=\sum_{k=1}^n (-1)^{k-1}(k-1)!c(n,k)$$
are always nonnegative, through a study of increasing trees, rank-generating functions, and two-player games. Here, $c(n,k)$ is a \textit{Stirling number of the first kind}, counting the number of permutations of $n$ with $k$ cycles. To see the equivalence between the two descriptions of $a_n$, use that for $\tau\vdash [n]$, the number $(\lambda(\tau)-1)!$ counts the number of permutations of $n$ whose cycles form the blocks of $\tau$. Note also that the numbers $a_n$ differ from the numbers in the previous paragraph by a sign of $(-1)^{n-1}$. We remark that from writing \cite[Equation 4.2]{BHRZ} and \cite[Theorem 3.4]{RosasSagan} in terms of M\"obius functions we have $a_n=(|\mu|\ast \mu)([n])$, where $\mu$ is the M\"obius function on the lattice of set partitions of $[n]$, and $\ast$ denotes convolution.

Our paper is structured as follows. In Section~\ref{sec:inc}, we compute the exponential generating function of $(a_n)_{n\ge 1}$, and relate the numbers to the rank-generating functions of the lattices of prunings of increasing trees. In Section~\ref{sec:tamari}, we give bijections between plane trees, $213$-avoiding permutations, and $312$-avoiding permutations, and show that the construction in the previous section leads to a lattice quotient from the weak order to the Tamari lattice in terms of plane trees. In Section~\ref{sec:game}, we relate the numbers $a_n$ to finite two-player games, and use this connection to prove combinatorial identities involving the numbers. Finally, in Section~\ref{sec:cell}, we construct a projective variety associated to a rooted tree, and show that the rooted tree is a game tree won by the first or second player depending on whether the Euler characteristic of the real points of the variety is $0$ or $1$.

\section{Increasing trees and generating functions}\label{sec:inc}

Our first order of business in our investigation of the numbers 
$$a_n=\sum_{k=1}^n(-1)^{k-1}(k-1)!c(n,k)$$
is to compute the associated exponential generating function. For an introduction to exponential generation functions, see \cite[Chapter 4]{AOC}.

\begin{proposition}
The exponential generating function of the sequence $(a_n)_{n\ge 1}$ is
$$\sum_{n=1}^\infty \frac{a_nx^n}{n!}=\log(1-\log(1-x)).$$
\end{proposition}

\begin{proof}
By \cite[Table 4.1]{AOC}, for positive integers $k$,
$$\sum_{n=1}^\infty \frac{c(n,k)x^n}{n!}=\frac{(-\log(1-x))^k}{k!}.$$
Therefore,
\begin{align*}
\sum_{n=1}^\infty \frac{a_nx^n}{n!}&=\sum_{k=1}^\infty \frac{(-1)^{k-1}(k-1)!(-\log(1-x))^k)}{k!}\\
&=\sum_{k=1}^\infty\frac{-(\log(1-x))^k}{k}=\log(1-\log(1-x)).
\end{align*}
\end{proof}

To work towards a combinatorial interpretation of the numbers $a_n$, we begin by constructing a signed set enumerated by the terms of the alternating sum described by $a_n$.

For any permutation of $n$ with $k$ cycles, the number $(k-1)!$ counts the number of reorderings of its cycle decomposition such that the cycle containing $1$ is first. When writing the cycle decomposition of a permutation, we will write the smallest number in each cycle first. As an example, the $(3-1)!=2$ reorderings of the cycle decomposition $(12)(364)(5)$ counted are itself and $(12)(5)(364)$.

Given one of these reorderings, we may write a permutation in one-line notation with separators to capture the order of the numbers appearing, with separators demarcating distinct cycles. For example, $(12)(5)(364)$ becomes $12|5|364$, with underlying permutation $125364$ in one-line notation. Therefore, if we let $W_n$ denote the signed set of permutations in $S_1\times S_{n-1}$ (permutations of $n$ fixing $1$) in one-line notation with separators such that the first element of each block is the smallest, and let the sign of such a permutation with separators be $(-1)$ to the number of separators, we see that
$$a_n=\sum_{x\in W_n}\sgn(x).$$
For example, the elements of $W_3$ with sign $+1$ are $\{123,132,1|2|3,1|3|2\}$, and the elements of $W_3$ with sign $-1$ are $\{1|23,12|3,13|2\}$, so $a_3=4-3=1$.

For a permutation $\alpha \in S_1\times S_{n-1}$, let $L_\alpha$ denote the partial order on elements of $W_n$ with underlying permutation $\alpha$, ordered by inclusion of separators. An \textit{inversion} of $\alpha$ is a pair $(i,j)$ such that $i<j$ and $\alpha(i)>\alpha(j)$. The \textit{first inversion} from position $i$ (if it exists) is the inversion $(i,j)$ such that there are no inversion $(i,j')$ with $j'<j$.

Let $\gamma(\alpha)$ denote the \textit{first inversion tree} of $\alpha$, which is a rooted tree with vertices labelled $1,\dots,n$ with root labelled $1$ and $j$ the parent of $i$ whenever $(\alpha^{-1}(i),\alpha^{-1}(j))$ is a first inversion or $j=1$ and there are no inversions from $\alpha^{-1}(i)$. By construction, $\gamma(\alpha)$ is an \textit{increasing tree}, meaning the label of a vertex is always greater than the label of its parent. Let $\inc(n)$ denote the set of increasing trees on vertices labelled $1,\dots,n$. We will draw rooted trees and increasing trees with each vertex above its children.

For a rooted tree $T$, a \textit{pruning} of $T$ is a rooted tree with the same root and edges given by a subset of the edges of $T$. Let $L_T$ denote the lattice of prunings of $T$, ordered by inclusion of edges. $L_T$ is a \textit{lattice} (every pair of elements $x,y$ has a least upper bound $x\vee y$, or \textit{join}, and a greatest lower bound $x\wedge y$, or \textit{meet}) and is \textit{distributive} (join and meet distribute over each other), because it is isomorphic to the lattice of order ideals of $\hat T$. Here, $\hat T$ is the partial order on the non-root vertices of $T$, where $v\le w$ if and only if $v$ is ancestor of $w$ or is $w$.


\begin{lemma}\label{lem:lattice}
For permutations $\alpha\in S_1\times S_{n-1}$, $L_\alpha\cong L_{\gamma(\alpha)}$.
\end{lemma}

\begin{proof}
Note a separator may be placed preceding position $i$ in $\alpha$ if and only if there are no inversions from $i$ or there is already a separator preceding position $j$, where $(i,j)$ is the first inversion from $i$. (If there is a separator preceding position $i$, a block containing position $j$ cannot start before position $j$ because $\alpha(i),\dots,\alpha(j-1)$ are all greater than $\alpha(j)$.) That is, the positions of separators in an element of $W_n$ with underlying permutation $\alpha$ must correspond to an order ideal of $\hat\gamma(\alpha)$. It follows, then, that $L_\alpha\cong L_{\gamma(\alpha)}$. 
\end{proof}

A \textit{plane tree} is a rooted tree with the children of each vertex ordered from left to right. Note every increasing tree is associated with a plane tree $\rho(T)$, by ordering the children of each vertex from left to right in order of increasing labels. The \textit{postorder traversal} of a plane tree traverses a plane tree recursively, visiting the subtrees of a vertex first from left to right, then the vertex itself.

\begin{lemma}\label{lem:inc}
The map $\gamma: S_1\times S_{n-1}\to \inc(n)$ is a bijection. For $T\in\inc(n)$, $\gamma^{-1}(T)\in S_1\times S_{n-1}$ is the permutation satisfying for $2\le i \le n$, $(\gamma^{-1}(T))(i)$ is the label in $T$ of the $(i-1)$th vertex in postorder traversal of $\rho(T)$.
\end{lemma}
\begin{proof}
Let $\alpha\in S_1\times S_{n-1}$. Note whenever the vertex $2\le i \le n$ is a parent of the vertex labelled $j$ in $\gamma(\alpha)$ that $(\alpha^{-1}(j),\alpha^{-1}(i))$ is a first inversion, so $j$ is left of $i$ in the one-line notation of $\alpha$. 

Now suppose two vertices with a common parent are labelled by $i<j$. Any descendant of $i$ is left of $i$ in $\alpha$ and any descendant of $j$ is left of $j$ in $\alpha$. Moreover, $i$ is left of $j$, or else the first inversion from position $\alpha^{-1}(j)$ is would be to either $\alpha^{-1}(i)$ or a position left of it, and so $i$ and $j$ could not have a common parent. 

We next show that all descendants of $j$ appear right of $i$ in $\alpha$. Since $\gamma(\alpha)$ is an increasing tree, any descendant $j'$ of $j$ must satisfy $j'>j$. Then $j'$ cannot be left of $i$, or else since $j'>i$, the parent of $j'$ must be left of $i$, which is impossible by induction.

Since $\alpha\in S_1\times S_{n-1}$, the one-line notation of $\alpha$ must exactly describe the labels of $\gamma(\alpha)$ in the postorder traversal of $\rho(\gamma(\alpha))$, except with $1$ moved from the end to the beginning. 

We've seen that $\gamma$ is injective. To see that it is bijective, note the cardinalities of $S_1\times S_{n-1}$ and $\inc(n)$ are both $(n-1)!$, as we can build an increasing tree inductively with $i-1$ choices for the parent of the vertex labelled $i$.
\end{proof}
\begin{example}
Below is the increasing tree $\gamma(1623574)$ drawn with children of a vertex drawn from left to right in order of increasing labels. The labels in $\gamma(1623574)$ of the postorder traversal of $\rho(\gamma(1623574))$ are $6\to 2\to 3\to 5\to 7\to 4\to 1$.
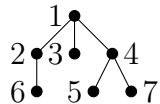
\begin{center}
\begin{tikzpicture}
\coordinate (1) at (0,0);
\coordinate (2) at (-.5,-.5);
\coordinate (3) at (-.5,-1);
\coordinate (4) at (0,-.5);
\coordinate (5) at (.5,-.5);
\coordinate (6) at (.25,-1);
\coordinate (7) at (.75,-1);

\filldraw [black] (1) circle (2pt);
\filldraw [black] (2) circle (2pt);
\filldraw [black] (3) circle (2pt);
\filldraw [black] (4) circle (2pt);
\filldraw [black] (5) circle (2pt);
\filldraw [black] (6) circle (2pt);
\filldraw [black] (7) circle (2pt);
\draw[-] (1) to (2) to (3);
\draw[-] (1) to (4);
\draw[-] (5) to (7);
\draw[-] (1) to (5) to (6);
\node[] at (-0.25,0) { $1$};
\node[] at (-0.75,-.5) { $2$};
\node[] at (0.75,-.5) { $4$};
\node[] at (-0.75,-1) { $6$};
\node[] at (-0.25,-.5) { $3$};
\node[] at (0,-1) { $5$};
\node[] at (1,-1) { $7$};
\end{tikzpicture}
\end{center}

\end{example}
Because $L_{\alpha}\cong L_{\gamma(\alpha)}$ is a finite distributive lattice, it has a \textit{rank-generating function} $\rgf_{L_{\alpha}}(q)=\rgf_{L_{\gamma(\alpha)}}(q)$ obtained by summing $q$ to the number of separators over all elements in $L_\alpha$. The following theorem then holds via Lemmas~\ref{lem:lattice} and \ref{lem:inc} because both sides enumerate the elements of $W_n$ with weight given by $q$ to the number of separators.

\begin{theorem}
For positive integers $n$,
    $$\sum_{k=1}^{n}(k-1)!c(n,k)q^{k-1}=\sum_{T\in\inc(n)}\rgf_{L_T}(q).$$
\end{theorem}

\begin{example}
For $n=3$, we have $0!c(3,1)+1!c(3,2)q+2!c(3,3)q^2 = (1+2q+q^2)+(1+q+q^2)$ from the following increasing trees.
\begin{center}
\begin{tikzpicture}
\coordinate (1) at (0,0);
\coordinate (2) at (-.25,-.5);
\coordinate (3) at (.25,-.5);
\filldraw [black] (1) circle (2pt);
\filldraw [black] (2) circle (2pt);
\filldraw [black] (3) circle (2pt);
\draw[-] (1) to (2);
\draw[-] (1) to (3);
\node[] at (-0.25,0) { $1$};
\node[] at (-0.5,-.5) { $2$};
\node[] at (0.5,-.5) { $3$};
\coordinate (1) at (4,0);
\coordinate (2) at (4,-.5);
\coordinate (3) at (4,-1);
\filldraw [black] (1) circle (2pt);
\filldraw [black] (2) circle (2pt);
\filldraw [black] (3) circle (2pt);
\draw[-] (1) to (2) to (3);
\node[] at (3.75,0) { $1$};
\node[] at (3.75,-.5) { $2$};
\node[] at (3.75,-1) { $3$};

\end{tikzpicture}
\end{center}
\end{example}

\begin{corollary} \label{cor:-1}
For positive integers $n$,
$$a_n=\sum_{T\in\inc(n)}\rgf_{L_T}(-1).$$
\end{corollary}
\section{Plane trees and a quotient to the Tamari lattice}\label{sec:tamari}

In this section, we make an excursion to show that $\rho\circ\gamma$ induces a lattice quotient from the weak order on $S_1\times S_{n-1}$ to the Tamari lattice described via plane trees. The \textit{weak order} is a lattice structure on $S_1\times S_{n-1}$ where $\alpha\le\beta$ if and only if the inversion set of $\alpha$ is contained in the inversion set of $\beta$. Write $\inv(\alpha)$ for the set of inversions of $\alpha$. A \textit{lattice quotient} is a surjective map of lattices that is order-preserving and preserves joins and meets. The Tamari lattice, first introduced by Tamari in his 1951 thesis, was shown to arise from a lattice quotient of the weak order in \cite{Reading}.

We begin by describing stack-based algorithms to label plane trees. We will write $\pl(n)$ to denote the set of plane trees on $n$ vertices. A \textit{stack} is a ``last in, first out'' data structure, analogous to a stack of books on a table. It has two operations: \textit{push}, which adds an element to the stack, and \textit{pop}, which removes the element in the stack added last. In this section, when we include an additional parameter $i$ to the push or pop operations of a stack, we will understand it to mean to label the vertex being pushed or popped with the label $i$.

Our first algorithm takes in the root of a plane tree and gives the \textit{eastpush-labelling} of the plane tree.

\begin{alg}$ $
\begin{algorithmic}
\State $\text{stack} \gets \text{[ ]}$
\State $i \gets 1$
\State stack.push(root, $i$)
\State $i\gets i+1$
\While {stack not empty}
\State $\text{vertex}\gets\text{stack.pop()}$
\For {child of vertex from left to right}
\State stack.push(child, $i$)
\State $i\gets i+1$
\EndFor
\EndWhile
\end{algorithmic}
\end{alg}

For this section, our analysis will be simplified if we consider the root of our plane tree to have a parent. That way, the notion of grandparent is well-defined for all vertices other than the root. We also extend our notion of left, so we say that $v$ is left of $w$ if and only if $v$ is contained in a subtree rooted at a vertex that is a left sibling of the root of a subtree containing $w$. Note $v$ appears before $w$ in the postorder traversal if and only if either $v$ is left of $w$ or $w$ is an ancestor of $v$.

\begin{lemma}\label{lem:east}
For $T\in \pl(n)$, the vertices labelled before a given vertex in the eastpush-labelling of $T$ are
\begin{enumerate}
    \item all vertices right of its parent,
    \item all children of non-parent ancestors, and
    \item all left siblings.
\end{enumerate}
\end{lemma}
\begin{proof}
    We proceed by induction. The statement is vacuously true for the root of $T$.

    For the inductive step, consider a vertex being pushed onto the stack. By the time its parent was pushed onto the stack, the vertices that had been labelled were all vertices right of the vertex's grandparent, all children of ancestors that are not the vertex's parent or grandparent, all left siblings of the parent, and the parent. 
    
    Because the vertex's grandparent's children were pushed from left to right, the parent is popped after all of its right siblings are popped. Between the parent being pushed and popped, the algorithm traverses all subtrees rooted at right siblings of the parent, and labels everything right of the vertex's parent but not right of the vertex's grandparent.

    Finally, when the parent is being processed, before the vertex is pushed, all left siblings of the vertex are pushed onto the stack. Putting this altogether, we see that the vertices labelled before the vertex is pushed onto the stack consist of everything right of its parent, all children of nonparent ancestors, and all left siblings.
\end{proof}

Our second algorithm constructs the \textit{westpop-labelling} of a plane tree.

\begin{alg}$ $
\begin{algorithmic}
\State $\text{stack} \gets \text{[ ]}$
\State $i \gets 1$
\State stack.push(root)
\While {stack not empty}
\State $\text{vertex}\gets\text{stack.pop($i$)}$
\State $i \gets i+1$
\For {child of vertex from right to left}
\State stack.push(child)
\EndFor
\EndWhile
\end{algorithmic}
\end{alg}

\begin{lemma}\label{lem:west}
For $T\in \pl(n)$, the vertices labelled before a given vertex in the westpop-labelling of $T$ are
\begin{enumerate}
\item all vertices to its left, and
\item all ancestors.
\end{enumerate}
\end{lemma}
\begin{proof}
    We proceed by induction. The statement is vacuously true for the root of $T$.

    For the inductive step, consider a vertex being popped off of the stack. By the time its parent was being processed, the vertices that had been labelled were all vertices left of the vertex's parent and all ancestors. While the parent was being processed, the vertex and its siblings are pushed onto the stack from right to left.

    Thus, before the vertex is popped, all left siblings and all vertices in any subtree rooted at them must be processed by the algorithm. So before the vertex is popped, the vertices that have already been labelled are all vertices to its left and all ancestors.
\end{proof}

By Lemmas~\ref{lem:east} and \ref{lem:west}, the eastpush- and westpop-labellings of a plane tree have labels increasing among children of a vertex from left to right and from parents to children. Let $\east,\west:\pl(n)\to\inc(n)$ denote the eastpush- and westpop-labellings of a plane tree, respectively. These labellings are related by $\gamma^{-1}$ to 213- and 312-avoiding permutations. A permutation $\alpha\in S_1\times S_{n-1}$ \textit{avoids $213$} (respectively $312$) if there exists indices $i<j<k$ such that $\alpha(i),\alpha(j),\alpha(k)$ are in the same relative order as $2,1,3$ (respectively, $3,1,2$).

\begin{proposition}
    For $T\in \pl(n)$,
    \begin{enumerate}
        \item $\gamma^{-1}(\east(T))$ is $213$-avoiding, and
        \item $\gamma^{-1}(\west(T))$ is $312$-avoiding.
    \end{enumerate}
\end{proposition}
\begin{proof}
For part (1), suppose for a contradiction that $i<j<k$ and the postorder traversal of $T$ visits the vertices labelled in $\east(T)$ by $j,i,k$ in that order. Since $k>i,j$ and $k$ is visited last among $i,j,k$ in the postorder traversal while $\east(T)$ is an increasing tree, $k$ must be right of $i$ and $j$. Since $i,j$ are labelled before $k$ in the eastpush-labelling of $T$, $i$ and $j$ must then be children of ancestors of $k$ by Lemma~\ref{lem:east}, but are not ancestors themselves. Thus neither $i$ nor $j$ is a descendant of the other, and so $j$ is left of $i$.

However, $i$ cannot be right of $j$'s parent, because $i$ is left of $k$ and $j$'s parent is an ancestor of $k$. $i$ cannot be a child of a non-parent ancestor of $j$, because then since $i$ is left of $k$, the child of the ancestor that is the root of the subtree containing $j$ must be right of $i$, and so $j$ would be right of $i$. And $i$ is not a left sibling of $j$. This contradicts Lemma~\ref{lem:east}.

For part (2), suppose for a contradiction that $i<j<k$ and the postorder traversal of $T$ visits the vertices labelled in $\west(T)$ by $k,i,j$, in that order. By Lemma~\ref{lem:west}, this requires that $i$ must be left of $j$, and $i,j$ are ancestors of $k$. This is already a contradiction, since $i$ must then either be an ancestor or a descendant of $j$.
\end{proof}
Because a class of numbers known as the Catalan numbers enumerate $213$- and $312$-avoiding permutations \cite[Section 2.2.1]{Knuth1}, as well as plane trees \cite{Cayley}, we have the following.
\begin{theorem}\label{thm:bij}
    The map $\rho\circ\gamma: S_1\times S_{n-1}\to \pl(n)$ restricts to
    \begin{enumerate}
        \item a bijection from $213$-avoiding permutations to plane trees, with inverse $\gamma^{-1}\circ\east$, and
        \item a bijection from $312$-avoiding permutations to plane trees, with inverse $\gamma^{-1}\circ\west$.
    \end{enumerate}
\end{theorem}
\begin{example}
Drawn are the eastpush-labelling on the left and the westpop-labelling on the right of the same plane tree. The left is $\gamma$ applied to the $213$-avoiding permutation $1723564$ and the right is $\gamma$ applied to the $312$-avoiding permutation $1324675$.
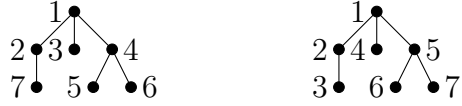
\begin{center}
\begin{tikzpicture}
\coordinate (1) at (0,0);
\coordinate (2) at (-.5,-.5);
\coordinate (3) at (-.5,-1);
\coordinate (4) at (0,-.5);
\coordinate (5) at (.5,-.5);
\coordinate (6) at (.25,-1);
\coordinate (7) at (.75,-1);

\filldraw [black] (1) circle (2pt);
\filldraw [black] (2) circle (2pt);
\filldraw [black] (3) circle (2pt);
\filldraw [black] (4) circle (2pt);
\filldraw [black] (5) circle (2pt);
\filldraw [black] (6) circle (2pt);
\filldraw [black] (7) circle (2pt);
\draw[-] (1) to (2) to (3);
\draw[-] (1) to (4);
\draw[-] (5) to (7);
\draw[-] (1) to (5) to (6);
\node[] at (-0.25,0) { $1$};
\node[] at (-0.75,-.5) { $2$};
\node[] at (0.75,-.5) { $4$};
\node[] at (-0.75,-1) { $7$};
\node[] at (-0.25,-.5) { $3$};
\node[] at (0,-1) { $5$};
\node[] at (1,-1) { $6$};

\coordinate (1) at (0+4,0);
\coordinate (2) at (-.5+4,-.5);
\coordinate (3) at (-.5+4,-1);
\coordinate (4) at (0+4,-.5);
\coordinate (5) at (.5+4,-.5);
\coordinate (6) at (.25+4,-1);
\coordinate (7) at (.75+4,-1);

\filldraw [black] (1) circle (2pt);
\filldraw [black] (2) circle (2pt);
\filldraw [black] (3) circle (2pt);
\filldraw [black] (4) circle (2pt);
\filldraw [black] (5) circle (2pt);
\filldraw [black] (6) circle (2pt);
\filldraw [black] (7) circle (2pt);
\draw[-] (1) to (2) to (3);
\draw[-] (1) to (4);
\draw[-] (5) to (7);
\draw[-] (1) to (5) to (6);
\node[] at (-0.25+4,0) { $1$};
\node[] at (-0.75+4,-.5) { $2$};
\node[] at (0.75+4,-.5) { $5$};
\node[] at (-0.75+4,-1) { $3$};
\node[] at (-0.25+4,-.5) { $4$};
\node[] at (0+4,-1) { $6$};
\node[] at (1+4,-1) { $7$};
\end{tikzpicture}
\end{center}
\end{example}



Next, we study how $213$- and $312$-avoiding permutations appear in the fibers of $\rho\circ\gamma$. For $\alpha\in S_1\times S_{n-1}$, the \textit{first inversion function} of $\alpha$ is the function $t_\alpha:\{2,\dots,n+1\}\to\{2,\dots,n+1\}$ defined by $t(n+1)=t(n+1)$, and for $2\le i \le n$, $t(i)=j$ if $(i,j)$ is the first inversion from position $i$ and $t(i)=n+1$ if there are no inversions from position $i$.

\begin{lemma}\label{lem:213}
For $213$-avoiding $\alpha\in S_1\times S_{n-1}$, $\inv(\alpha)=\{(i,j)\mid t_\alpha(i)\le j \le n\}$.
\end{lemma}
\begin{proof}
    If $t_\alpha(i)\le n$, then $(i,t_\alpha(i))\in\inv(\alpha)$ implies we further have $(i,j)\in \inv(\alpha)$ for all $t_\alpha(i)< j \le n$, or else we have a $213$-pattern in positions $i,t_\alpha(i),j$. By definition of $t_\alpha(i)$, there are no other inversions in $\alpha$.
\end{proof}

\begin{lemma}\label{lem:312}
For $312$-avoiding $\alpha\in S_1\times S_{n-1}$, $\inv(\alpha)=\{(i,j)\mid \{t^k_\alpha(i)\}_{k\ge 1}\ni j \le n\}$.
\end{lemma}
\begin{proof}
    The set $\{(i,j)\mid \{t^k_\alpha(i)\}_{k\ge 1}\ni j \le n\}$ consists of all inversions of $\alpha$ that are guaranteed by the first inversion function of $\alpha$. Suppose for a contradiction that these were not all, that there were some other inversion $(i,j)$ with $j\not\in \{t_\alpha^k(i)\}_{k\ge 1}$. Then $(t_\alpha(i),j)\in \inv(G)$, or else there is a $312$-pattern in positions $i,t_{\alpha(i)},j$. Repeating this argument, we find that there is some iterate $t_\alpha^k(i)$ such that $(t_\alpha^k(i),j)\in \inv(\alpha)$, but $j<t_\alpha^{k+1}(i)$, contradicting that $t_\alpha$ keeps track of first inversions.
\end{proof}

\begin{proposition} \label{prop:int}
    For $T\in\pl(n)$, $\gamma^{-1}(\rho^{-1}(T))$ is an interval in the weak order on $S_1\times S_{n-1}$, with top element $213$-avoiding and bottom element $312$-avoiding.
\end{proposition}
\begin{proof}
For all $\alpha\in \gamma^{-1}(\rho^{-1}(T))$, $t_\alpha(i)=j$ for $2\le i\le n$ if the $(i-1)$th vertex in the postorder traversal of $T$ is the $(j-1)$th vertex in the traversal. By Theorem~\ref{thm:bij}, there is exactly one $213$-avoiding permutation and one $312$-avoiding permutation in the fiber of $\rho\circ\gamma$ containing $\alpha$. By Lemmas~\ref{lem:213} and \ref{lem:312}, their inversion sets are $\{(i,j)\mid t_\alpha(i)\le j\le n\}$ and $\{(i,j)\mid \{t_\alpha^k(i)\}_{k\ge1}\ni j \le n\}$, respectively. These are, respectively, all possible inversions given first inversion function $t_\alpha$, and all guaranteed inversions given first inversion function $t_\alpha$. Hence $\alpha$ lies between them in the weak order in $S_1\times S_{n-1}$.
\end{proof}

We call the equivalence relation induced on the fibers of a lattice quotient a \textit{lattice congruence}.

\begin{proposition}\cite[Section 3]{Reading}\label{prop:cong}
    An equivalence relation on a lattice $L$ is a lattice congruence if and only if all of the following hold:
    \begin{enumerate}
        \item each equivalence class is an interval,
        \item the map $\pi^\uparrow$ from $L$ sending an element to the top of its equivalence class is order-preserving, and
        \item the map $\pi_\downarrow$ from $L$ sending an element to the bottom of its equivalence class is order-preserving.
    \end{enumerate}
\end{proposition}

\begin{theorem}
    The map $\rho\circ \gamma: S_1\times S_{n-1}\to \pl(n)$ induces the Tamari lattice structure on $\pl(n)$ respecting the lattice structure of the weak order on $S_1\times S_{n-1}$, making $\rho\circ\gamma$ a lattice quotient.
\end{theorem}
\begin{proof}
    We will show that the fibers of $\rho\circ\gamma$ form a lattice congruence. That the lattice structure on plane trees is the Tamari lattice then follows from \cite[Section 9]{BW}, where they identify the lattice on $213$-avoiding permutations in the weak order with the Tamari lattice.

    Condition (1) of Proposition~\ref{prop:cong} is Proposition~\ref{prop:int}. For condition (2), let $\alpha\le\beta\in S_1\times S_{n-1}$ and let $\pi^\uparrow(\alpha),\pi^{\uparrow}(\beta)$ denote the greatest elements in the fibers of $\rho\circ\gamma$ containing $\alpha,\beta$, respectively. Since $\alpha\le\beta$ in the weak order, we see that each $t_\beta(i)\le t_\alpha(i)$. Therefore, $\{(i,j)\mid t_\alpha(i)\le j \le n\}\subseteq \{(i,j)\mid t_\beta(i)\le j \le n\}$, and so $\pi^\uparrow(\alpha)\le \pi^\uparrow(\beta)$ in the weak order. The proof of condition (3) is similar. (To see that $t_\alpha(i)$ lies in $\{t^k_\beta(i)\}_{k\ge 1}$, note for all $i\le j <t_\alpha(i)$ that $t_\beta(j)\le t_\alpha(j)\le t_\alpha(i).$)
\end{proof}

\begin{remark}
    In terms of first inversion functions, the join and meet operators of the Tamari lattice can deduced from $(t_{\alpha\vee\beta})(i)=\min(t_\alpha(i),t_\beta(i))$ and $(t_{\alpha\wedge \beta})(i)=\min (\{t_\alpha^k(i)\}_{k\ge 1}\cap \{t_\beta^k(i)\}_{k\ge 1})$. First inversion functions are shifts of the bracketing functions studied by Huang and Tamari in \cite{HuangTamari}, but the meet is not so easily described using bracketing functions, and so was omitted from their paper.
    
    The description of the Tamari lattice in terms of plane trees can be shown to be equivalent to Knuth's description in terms of ordered forests (plane trees with the root removed) in \cite[Section 7.2.1.6]{Knuth4A}. The description of the quotient from the weak order to the Tamari lattice described using plane trees via first inversions appears to be new.
\end{remark}
\section{Finite two-player games of perfect information}\label{sec:game}
In this section, we will connect the numbers $a_n$ to two-player games and use this connection to prove identities involving $a_n$. We begin by defining a polynomial $\varphi_T(q)$ associated to a rooted tree $T$.

\begin{definition}
    For a rooted tree $T$, $\varphi_T(q)$ is defined recursively by the following rule: if $T$ has children which are roots of the subtrees $T_1,\dots,T_m$, then
    $$\varphi_T(q)=\prod_{k=1}^m (1+q\varphi_{T_k}(q)).$$
\end{definition}
Here, we take $1$ to be the empty product.
\begin{remark}
    In general, $\varphi_T(q)$ does not distinguish rooted trees and is not unimodal. For the following rooted trees $T_1,T_2,T_3$, $\varphi_{T_1}(q)=\varphi_{T_2}(q)=1+2q+3q^2+4q^3+4q^4+3q^5+q^6$ and $\varphi_{T_3}(q)=1+3q+3q^2+6q^3+10q^4+11q^5+10q^6+11q^7+10q^8+5q^9+q^{10}$.
    
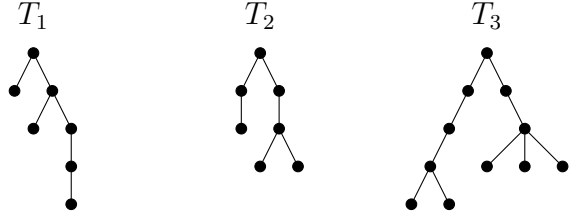
\begin{center}
\begin{tikzpicture}
\coordinate (1) at (0,0);
\coordinate (2) at (-.25,-.5);
\coordinate (3) at (.25,-.5);
\coordinate (4) at (0, -1);
\coordinate (5) at (.5,-1);
\coordinate (6) at (.5,-1.5);
\coordinate (7) at (.5,-2);
\filldraw [black] (1) circle (2pt);
\filldraw [black] (2) circle (2pt);
\filldraw [black] (3) circle (2pt);
\filldraw [black] (4) circle (2pt);
\filldraw [black] (5) circle (2pt);
\filldraw [black] (6) circle (2pt);
\filldraw [black] (7) circle (2pt);
\draw[-] (1) to (2);
\draw[-] (1) to (3) to (4);
\draw[-] (3) to (5) to (6) to (7);
\node[] at (0,.5) { $T_1$};

\coordinate (1) at (0+3,0);
\coordinate (2) at (-.25+3,-.5);
\coordinate (3) at (-.25+3,-1);
\coordinate (4) at (.25+3, -.5);
\coordinate (5) at (.25+3,-1);
\coordinate (6) at (0+3,-1.5);
\coordinate (7) at (.5+3,-1.5);
\filldraw [black] (1) circle (2pt);
\filldraw [black] (2) circle (2pt);
\filldraw [black] (3) circle (2pt);
\filldraw [black] (4) circle (2pt);
\filldraw [black] (5) circle (2pt);
\filldraw [black] (6) circle (2pt);
\filldraw [black] (7) circle (2pt);
\draw[-] (1) to (2) to (3);
\draw[-] (1) to (4) to (5) to (6);
\draw[-] (5) to (7);
\node[] at (0+3,.5) { $T_2$};

\coordinate (1) at (6,0);
\coordinate (2) at (5.75,-.5);
\coordinate (3) at (5.5,-1);
\coordinate (4) at (5.25,-1.5);
\coordinate (5) at (5,-2);
\coordinate (6) at (5.5,-2);
\coordinate (7) at (6.25,-.5);
\coordinate (8) at (6.5,-1);
\coordinate (9) at (6,-1.5);
\coordinate (a) at (6.5,-1.5);
\coordinate (b) at (7,-1.5);
\filldraw [black] (1) circle (2pt);
\filldraw [black] (2) circle (2pt);
\filldraw [black] (3) circle (2pt);
\filldraw [black] (4) circle (2pt);
\filldraw [black] (5) circle (2pt);
\filldraw [black] (6) circle (2pt);
\filldraw [black] (7) circle (2pt);
\filldraw [black] (8) circle (2pt);
\filldraw [black] (9) circle (2pt);
\filldraw [black] (a) circle (2pt);
\filldraw [black] (b) circle (2pt);
\draw[-] (1) to (2) to (3) to (4) to (5);
\draw[-] (4) to (6);
\draw[-] (1) to (7) to (8) to (9);
\draw[-] (8) to (a);
\draw[-] (8) to (b);
\node[] at (6,.5) { $T_3$};
\end{tikzpicture}
\end{center}
\end{remark}
A rooted tree $T$ may be thought of as the \textit{game tree} of a finite two-player game of perfect information and alternating moves, where the children of each vertex denote the possible moves from that game state. We do not allow ties and say a player loses if they have no available moves. In this setting, Zermelo showed in \cite{Zermelo} that for every such game, either the first player or the second player has a winning strategy. We show that determining the player with the winning strategy is equivalent to evaluating $\varphi_T(-1)$.
\begin{lemma}\label{lem:win}
    For a rooted tree $T$, $\varphi_T(-1)=1$ if $T$ is the game tree of a finite two-player game of perfect information in which the second player has a winning strategy, otherwise $\varphi_T(q)=0$.
\end{lemma}
\begin{proof}
We proceed by induction on the size of the rooted tree. Suppose the inductive hypothesis holds for all rooted trees with fewer vertices than $T$. 

Suppose $T$ has children which are roots of the subtrees $T_1,\dots, T_m$. The root of $T$ is a losing position (and hence a player $2$ win) if and only if the roots of the $T_k$ are all winning positions. By the inductive hypothesis, this occurs if and only if each $\varphi_{T_k}(-1)=0$, and so
$$\varphi_T(-1)=\prod_{k=1}^m(1+(-1)\cdot 0)=1.$$

Otherwise, the root of some $T_k$ is a losing position, and so $(-\varphi_{T_k}(-1)+1)=0$ by the inductive hypothesis, implying $\varphi_T(-1)=0$.
\end{proof}

Next, we relate $\varphi_T(q)$ to the rank-generating functions encountered in Section~\ref{sec:inc}.

\begin{proposition}\label{prop:gamergf}
    For a rooted tree $T$, $\varphi_T(q)=\rgf_{L_T}(q)$.
\end{proposition}
    
\begin{proof}
    We will show that $\rgf_{L_T}(q)$ satisfies the same recursion defining $\varphi_T(q)$.

    Suppose $T$ has children which are roots of the subtrees $T_1,\dots, T_m$. Let $T_1',\dots, T_m'$ denote the rooted trees consisting of the root of $T$ connected to each of $T_1,\dots, T_m$, respectively. The data of a pruning of $T$ is exactly the data of a pruning of each of the $T_k'$, so $L_T$ is isomorphic to the Cartesian product $L_{T_1'}\times\cdots\times L_{T_m'}$. Each lattice $L_{T_k'}$ is exactly the lattice of $L_{T_k}$ with a new minimum element adjoined.

    Hence
    $$\rgf_{L_T}(q)=\prod_{k=1}^m \rgf_{L_{T_k'}}(q)=\prod_{k=1}^m (1+q\rgf_{L_{T_k}}(q)),$$
    as desired.
\end{proof}

Combining Corollary~\ref{cor:-1}, Lemma~\ref{lem:win}, and Proposition~\ref{prop:gamergf} gives the following.

\begin{theorem}\label{thm:pos}
    The integer $a_n$ counts the number of increasing trees on vertices labelled $1,\dots,n$ that are game trees in which the second player has a winning strategy. In particular, each $a_n$ is nonnegative.
\end{theorem}
\begin{example}
$a_5=8$, from the following increasing trees. 
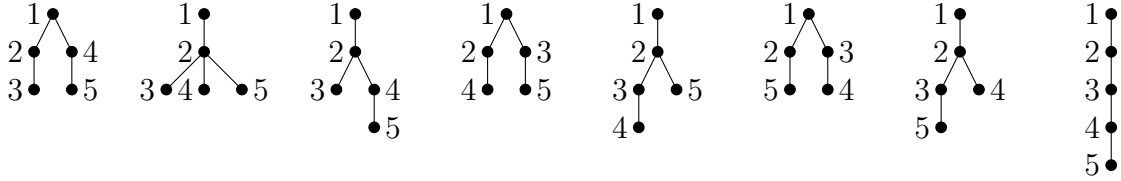
\begin{center}
\begin{tikzpicture}
\coordinate (1) at (0,0);
\coordinate (2) at (-.25,-.5);
\coordinate (3) at (-.25,-1);
\coordinate (4) at (.25, -.5);
\coordinate (5) at (.25,-1);
\filldraw [black] (1) circle (2pt);
\filldraw [black] (2) circle (2pt);
\filldraw [black] (3) circle (2pt);
\filldraw [black] (4) circle (2pt);
\filldraw [black] (5) circle (2pt);
\draw[-] (1) to (2) to (3);
\draw[-] (1) to (4) to (5);
\node[] at (-0.25,0) { $1$};
\node[] at (-0.5,-.5) { $2$};
\node[] at (-0.5,-1) { $3$};
\node[] at (.5,-.5) { $4$};
\node[] at (.5,-1) { $5$};
\coordinate (1) at (2,0);
\coordinate (2) at (2,-.5);
\coordinate (3) at (1.5,-1);
\coordinate (4) at (2, -1);
\coordinate (5) at (2.5,-1);
\filldraw [black] (1) circle (2pt);
\filldraw [black] (2) circle (2pt);
\filldraw [black] (3) circle (2pt);
\filldraw [black] (4) circle (2pt);
\filldraw [black] (5) circle (2pt);
\draw[-] (1) to (2) to (3);
\draw[-] (2) to (4);
\draw[-] (2) to (5);
\node[] at (1.75,0) { $1$};
\node[] at (1.75,-.5) { $2$};
\node[] at (1.25,-1) { $3$};
\node[] at (1.75,-1) { $4$};
\node[] at (2.75,-1) { $5$};
\coordinate (1) at (4,0);
\coordinate (2) at (4,-.5);
\coordinate (3) at (3.75,-1);
\coordinate (4) at (4.25, -1);
\coordinate (5) at (4.25,-1.5);
\filldraw [black] (1) circle (2pt);
\filldraw [black] (2) circle (2pt);
\filldraw [black] (3) circle (2pt);
\filldraw [black] (4) circle (2pt);
\filldraw [black] (5) circle (2pt);
\draw[-] (1) to (2) to (3);
\draw[-] (2) to (4) to (5);
\node[] (1) at (3.75,0) {$1$};
\node[] (2) at (3.75,-.5) {$2$};
\node[] (3) at (3.5,-1) {$3$};
\node[] (4) at (4.5, -1) {$4$};
\node[] (5) at (4.5,-1.5) {$5$};
\coordinate (1) at (0+6,0);
\coordinate (2) at (-.25+6,-.5);
\coordinate (3) at (-.25+6,-1);
\coordinate (4) at (.25+6, -.5);
\coordinate (5) at (.25+6,-1);
\filldraw [black] (1) circle (2pt);
\filldraw [black] (2) circle (2pt);
\filldraw [black] (3) circle (2pt);
\filldraw [black] (4) circle (2pt);
\filldraw [black] (5) circle (2pt);
\draw[-] (1) to (2) to (3);
\draw[-] (1) to (4) to (5);
\node[] at (-0.25+6,0) { $1$};
\node[] at (-0.5+6,-.5) { $2$};
\node[] at (-0.5+6,-1) { $4$};
\node[] at (.5+6,-.5) { $3$};
\node[] at (.5+6,-1) { $5$};
\coordinate (1) at (8,0);
\coordinate (2) at (8,-.5);
\coordinate (3) at (7.75,-1);
\coordinate (4) at (7.75, -1.5);
\coordinate (5) at (8.25,-1);
\filldraw [black] (1) circle (2pt);
\filldraw [black] (2) circle (2pt);
\filldraw [black] (3) circle (2pt);
\filldraw [black] (4) circle (2pt);
\filldraw [black] (5) circle (2pt);
\draw[-] (1) to (2) to (3) to (4);
\draw[-] (2) to (5);
\node[] (1) at (7.75,0) {$1$};
\node[] (2) at (7.75,-.5) {$2$};
\node[] (3) at (7.5,-1) {$3$};
\node[] (4) at (7.5, -1.5) {$4$};
\node[] (5) at (8.5,-1) {$5$};
\coordinate (1) at (0+10,0);
\coordinate (2) at (-.25+10,-.5);
\coordinate (3) at (-.25+10,-1);
\coordinate (4) at (.25+10, -.5);
\coordinate (5) at (.25+10,-1);
\filldraw [black] (1) circle (2pt);
\filldraw [black] (2) circle (2pt);
\filldraw [black] (3) circle (2pt);
\filldraw [black] (4) circle (2pt);
\filldraw [black] (5) circle (2pt);
\draw[-] (1) to (2) to (3);
\draw[-] (1) to (4) to (5);
\node[] at (-0.25+10,0) { $1$};
\node[] at (-0.5+10,-.5) { $2$};
\node[] at (-0.5+10,-1) { $5$};
\node[] at (.5+10,-.5) { $3$};
\node[] at (.5+10,-1) { $4$};
\coordinate (1) at (8+4,0);
\coordinate (2) at (8+4,-.5);
\coordinate (3) at (7.75+4,-1);
\coordinate (4) at (7.75+4, -1.5);
\coordinate (5) at (8.25+4,-1);
\filldraw [black] (1) circle (2pt);
\filldraw [black] (2) circle (2pt);
\filldraw [black] (3) circle (2pt);
\filldraw [black] (4) circle (2pt);
\filldraw [black] (5) circle (2pt);
\draw[-] (1) to (2) to (3) to (4);
\draw[-] (2) to (5);
\node[] (1) at (7.75+4,0) {$1$};
\node[] (2) at (7.75+4,-.5) {$2$};
\node[] (3) at (7.5+4,-1) {$3$};
\node[] (4) at (7.5+4, -1.5) {$5$};
\node[] (5) at (8.5+4,-1) {$4$};
\coordinate (1) at (14,0);
\coordinate (2) at (14,-.5);
\coordinate (3) at (14,-1);
\coordinate (4) at (14, -1.5);
\coordinate (5) at (14,-2);
\filldraw [black] (1) circle (2pt);
\filldraw [black] (2) circle (2pt);
\filldraw [black] (3) circle (2pt);
\filldraw [black] (4) circle (2pt);
\filldraw [black] (5) circle (2pt);
\draw[-] (1) to (2) to (3) to (4) to (5);
\node[] (1) at (13.75,0) {$1$};
\node[] (2) at (13.75,-.5) {$2$};
\node[] (3) at (13.75,-1) {$3$};
\node[] (4) at (13.75, -1.5) {$4$};
\node[] (5) at (13.75,-2) {$5$};
\end{tikzpicture}
\end{center}
\end{example}

Thus, we have solved the main problem studied in this paper. We next give a formula for $\varphi_T(q)$ in terms of the prunings of $T$ for which the second player has a winning strategy.

\begin{theorem}
    For a rooted tree $T$,
    $$\varphi_T(q)=\sum_{\substack{S\in L_T\\ \varphi_S(-1)=1}} (-q)^{r(S)}(1+q)^{c(S)},$$
    where for a pruning $S$ of $T$, $r(S)$ is the rank of $S$ in $L_T$ and $c(S)$ is the number of elements of $L_T$ that cover $S$.
\end{theorem}
\begin{proof}
    Consider $q\in [-1, 0]$, i.e. $-q$ is a probability. Suppose we have a coin that flips heads with probability $-q$ and tails with probability $1+q$, and consider the random event $A_T(q)$ defined recursively on rooted trees as follows. If $T$ has children which are roots of the subtrees $T_1,\dots,T_m$, $A_T(q)$ occurs if and only if for each $k$ from $1$ to $m$, we flip our coin and it lands tails or, in the case it lands heads, $A_{T_k}(q)$ does not occur.

    Since the recursive relation computing $\text{P}(A_T(q))$, i.e.
    $$\text{P}(A_T(q)) = \prod_{k=1}^m (1+q \text{P}(A_{T_k}(q))),$$
    is the same relation defining $\varphi_T(q)$, we see that for all $q\in [-1,0]$, $\text{P}({A_T}(q))=\varphi_T(q)$. In a random trial, consider the pruning of $T$ consisting of all vertices visited in the recursive procedure, i.e. the root and all vertices for which the coin came up heads. For a pruning $S$ of $T$ to be the generated pruning, each of the coin flips corresponding to the $r(S)$ non-root vertices of $S$ must have landed heads, and each of the coin flips corresponding to the $c(S)$ vertices outside of $S$ that are incident to a vertex of $S$ must have landed tails. So $S$ is the generated pruning in a random trial with probability $(-q)^{r(S)}(1+q)^{c(S)}$.

    One also sees recursively that $A_T(q)$ occurs if and only if the generated pruning has a winning strategy for the second player. Thus, for $q\in [-1,0]$,
    $$\text{P}(A_T(q))=\sum_{\substack{S\in L_T\\\varphi_S(-1)=1}}(-q)^{r(S)}(1+q)^{c(S)}.$$

    Hence $\varphi_T(q)$ satisfies the given formula for all $q\in[-1,0]$. By polynomial interpolation, they agree for all $q$.
\end{proof}

\begin{example}
For the following rooted tree $T$,
$$\varphi_T(q)=(1+q)^2+q^2(1+q)^2+q^2(1+q)^2-q^3(1+q),$$
from the depicted prunings.
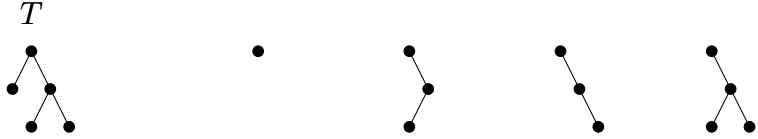
\begin{center}
    \begin{tikzpicture}
\coordinate (1) at (0,0);
\coordinate (2) at (-.25,-.5);
\coordinate (3) at (.25,-.5);
\coordinate (4) at (0, -1);
\coordinate (5) at (.5,-1);

\filldraw [black] (1) circle (2pt);
\filldraw [black] (2) circle (2pt);
\filldraw [black] (3) circle (2pt);
\filldraw [black] (4) circle (2pt);
\filldraw [black] (5) circle (2pt);

\draw[-] (1) to (2);
\draw[-] (1) to (3) to (4);
\draw[-] (3) to (5);
\node[] at (0,.5) { $T$};

\coordinate (1) at (0+3,0);

\filldraw [black] (1) circle (2pt);

\coordinate (1) at (0+3+2,0);
\coordinate (3) at (.25+5,-.5);
\coordinate (4) at (0+5, -1);

\filldraw [black] (1) circle (2pt);
\filldraw [black] (3) circle (2pt);
\filldraw [black] (4) circle (2pt);

\draw[-] (1) to (3) to (4);

\coordinate (1) at (0+7,0);
\coordinate (3) at (.25+7,-.5);
\coordinate (5) at (.5+7,-1);

\filldraw [black] (1) circle (2pt);
\filldraw [black] (3) circle (2pt);
\filldraw [black] (5) circle (2pt);

\draw[-] (1) to (3);
\draw[-] (3) to (5);

\coordinate (1) at (0+9,0);

\coordinate (3) at (.25+9,-.5);
\coordinate (4) at (0+9, -1);
\coordinate (5) at (.5+9,-1);

\filldraw [black] (1) circle (2pt);
\filldraw [black] (3) circle (2pt);
\filldraw [black] (4) circle (2pt);
\filldraw [black] (5) circle (2pt);

\draw[-] (1) to (3) to (4);
\draw[-] (3) to (5);
\node[] at (0,.5) { $T$};
\end{tikzpicture}
\end{center}
\end{example}

We conclude this section by applying Theorem~\ref{thm:pos} to prove recurrence relations for the numbers $a_n$.

\begin{proposition}
    For integers $n\ge 2$, the numbers $a_n$ satisfy the recurrence
    $$a_n=\sum_{k=1}^{n-1}{n-2\choose k-1}((k-1)!-a_{k})a_{n-k}.$$
\end{proposition}
\begin{proof}
For an increasing tree on $n$ vertices to have a winning strategy for player $2$, it is equivalent that the subtree rooted at the vertex labelled $2$ have its root be a winning position and the pruning of $T$ obtained by removing the subtree rooted at $2$ be a losing position. 

When the subtree rooted at vertex $2$ has $k$ vertices, there are ${n-2\choose k-1}$ choices for the set of labels of its vertices. There are $(k-1)!-a_k$ choices of an increasing tree with those labels that have its root be a winning position, and there are $a_{n-k}$ choices for the pruning to be a losing position. Summing the possibilities for all $k$ from $1$ to $n-1$ gives the identity.
\end{proof}
\begin{proposition}
    For positive integers $n$, the numbers $a_n$ satisfy the recurrence
    $$(n-1)!-a_n=\sum_{k=1}^{n-1}{n-1\choose k-1}(n-k-1)!a_{k}.$$
\end{proposition}
\begin{proof}
The left hand side counts the increasing trees on vertices labelled $1,\dots,n$ in which player $1$ has a winning strategy. We will show that there is a bijection to pairs of increasing trees whose labels have disjoint union $1,\dots,n$, the first of which has root labelled $1$ and is a losing position. This is sufficient to prove the identity, because if the first tree has $k$ vertices, there are ${n-1\choose k-1}$ choices for its labels, $(n-k-1)!$ choices for the second increasing tree, and $a_k$ choices for the first increasing tree to have its root be a losing position.

Consider the least label $i$ among losing positions that are children of the root of the increasing tree. Construct the pair comprised by the pruning consisting of the root connected to the subtrees rooted at vertices labelled $\le i$ among the children of the root, and the pruning consisting of all remaining edges. Since the children of the root are all $>i$ in the second pruning, we can change the label of its root to $i$ while keeping the increasing property. In the first pruning, we delete the edge connecting $1$ to $i$ and identify the two vertices, retaining the label $1$. The first increasing tree is now rooted at a losing position. This completes the construction of the bijection.

To undo this process, suppose we instead start with a pair of increasing trees whose labels have disjoint union $1,\dots,n$, the first of which has root labelled $1$ and is a losing position. Let the label of the root of the second increasing tree be $i$. In the first increasing tree, take all children of the root with labels $>i$ and make it so that they are now children of a new vertex labelled $i$, which in turn will be a child of the root. Then identify the roots of the new pair of increasing trees, retaining the label $1$. We obtain an increasing tree with labels $1,\dots,n$ that is rooted at a winning position.
\end{proof}

\section{Cell complexes associated to rooted trees}\label{sec:cell}
In our final section, we construct a projective variety associated to a rooted tree $T$, and show that over $\mathbb R$ and $\mathbb C$, it has a cell structure related to the lattice $L_T$. We will show that over $\mathbb R$, the Euler characteristic of this variety tells us whether $T$ is a game tree in which the first or second player has a winning strategy.

We begin by defining the projective variety $X_T$ associated with a rooted tree over an arbitrary field. 

\begin{definition}
    The projective variety $X_T\subseteq \mathbb P^{\varphi_T(1)-1}$ is defined recursively as follows:
    \begin{enumerate}
        \item if $T$ consists of a single point, $X_T=\mathbb P^0$,
        \item if $T$ has a single child which is the root of subtree $T_1$, is carved out by the equations of $X_{T_1}\subseteq \mathbb P^{\varphi_{T_1}(1)-1}$ applied to the first $\varphi_T(1)-1$ coordinates, and
        \item if $T$ consists of multiple rooted trees $T_1',\dots, T_m'$ joined at their roots, where the root of each $T_k'$ has a single child, then $X_T=X_{T_1'}\times\cdots\times X_{T_m'}$, via the Segre embedding $\mathbb P^{\varphi_{T_1'}(1)-1}\times\cdots\times \mathbb P^{\varphi_{T_m'}(1)-1}\subseteq \mathbb P^{\varphi_T(1)-1}.$
    \end{enumerate}
\end{definition}
\begin{proposition}
    $X_T$ may be written as a disjoint union of affine cells $\bigcup_{x\in L_T} X_x$ such that
    \begin{enumerate}
        \item the dimension of $X_x$ is the rank of $x$ in $L_T$, i.e. $X_x\cong \mathbb A^{\rk(x)}$, and
        \item $X_x(\mathbb R),X_x(\mathbb C)$ is in the closure $X_y(\mathbb R),X_y(\mathbb C)$, respectively, if and only if $x\le y$ in $L_T$.
    \end{enumerate}
\end{proposition}
\begin{proof}
We proceed by inductively. If $T$ consists of a single vertex, then $X_T=\mathbb P^0$, which is a single affine cell $\mathbb A^0$.

If the root of $T$ has a single child, the root of $T_1$, $X_T$ consists of all nonzero points $(\mathbf a:b)\in \mathbb P^{\varphi_T(1)-1}$ such that $\mathbf a$ satisfies the equations carving out $X_{T_1}$ in $\mathbb P^{\varphi_{T_1}(1)-1}$. When $\mathbf a\neq \mathbf 0$, the points can be given the product cell decomposition of $X_{T_1}\times\mathbb A^1$, and the inclusion relations between cells and closures of cells over $\mathbb R$ and $\mathbb C$ is inherited. When $\mathbf a=\mathbf 0$, we have a single point $(\mathbf 0:1)$, which is a cell $\mathbb A^0$, and it lies in the closure of every cell when $X_T$ is taken over $\mathbb R$ or $\mathbb C$. The inductive hypothesis is then satisfied because $L_T$ is the lattice $L_{T_1}$ with a new minimum element adjoined.

Otherwise, suppose $T$ consists of multiple rooted trees $T_1',\dots, T_m'$ joined at their roots, where the root of each $T_k'$ has a single child. Then we can take the product cell structure of $T_1'\times\cdots\times T_m'$, which is consistent with $L_T\cong L_{T_1'}\times\cdots\times L_{T_m'}$.
\end{proof}
\begin{corollary}\label{cor:geo}
For rooted trees $T$,
    \begin{enumerate}
        \item $\varphi_T(q)$ counts the number of points in $X_T(\mathbb F_q)$ whenever $q$ is a prime power,
        \item $\varphi_T(-1)=\chi(X_T(\mathbb R))$, and
        \item $\varphi_T(1) = \chi(X_T(\mathbb C))$.
    \end{enumerate}
\end{corollary}
\begin{proof}
    Part (1) holds because each $X_x(\mathbb F_q)$ has $q^{\rk(x)}$ points. For parts (2) and (3), note our paving by affines makes $X_T(\mathbb R)$ and $X_T(\mathbb C)$ finite cell complexes with $X_x(\mathbb R),X_x(\mathbb C)$ cells of real dimensions $\rk(x),2\rk(x)$, respectively. The result follows since the Euler characteristic of a finite cell complex is the alternating sum of the number of cells in each real dimension.
\end{proof}
\begin{remark}
    $\varphi_T(q^2)$ is the Poincar\'e polynomial of $X_T(\mathbb C)$ but note in general $\varphi_T(q)$ is not the Poincar\'e polynomial of $X_T(\mathbb R)$.
\end{remark}
Combining Lemma~\ref{lem:win}, Proposition~\ref{prop:gamergf}, and Corollary~\ref{cor:geo} gives the following.
\begin{theorem}
    For a rooted tree $T$, either
    \begin{enumerate}
        \item Player $1$ has a winning strategy and $\rgf_{L_T}(-1)=\chi(X_T(\mathbb R)) = 0$, or
        \item Player $2$ has a winning strategy and $\rgf_{L_T}(-1)=\chi(X_T(\mathbb R))=1$.
    \end{enumerate}
\end{theorem}
\begin{remark}
    The value of $\rgf_{L_T}(-1)$ depends only on the parity of the number of prunings of $T$. If ever challenged to a finite two-player game of perfect information and alternating moves, with the opponent making the mistake of granting the choice to go first or second, one can simply count the number of prunings of the game tree or compute the Euler characteristic of associated real variety to inform their decision.
\end{remark}

\section{Acknowledgments}\label{sec:ack}  The author would like to thank God for inspiration.

\bibliographystyle{plain}

\end{document}